\documentclass[a4paper,12pt]{amsart}
\usepackage{amsmath}
\usepackage{amssymb}
\usepackage{ifthen}
\usepackage{graphicx}
\usepackage{float}
\usepackage{cite}
\usepackage{amsfonts}
\usepackage{amscd}
\usepackage{amsxtra}
\usepackage{color}
\usepackage{geometry}
\newtheorem{theorem}{Theorem}[section]
\newtheorem{lemma}[theorem]{Lemma}
\newtheorem{corollary}[theorem]{Corollary}

\newtheorem{definition}[theorem]{Definition}

\theoremstyle{definition}

\numberwithin{equation}{section}

\def\be{\begin{equation}}
\def\ee{\end{equation}}

\newcounter{alphabet}


\begin{document}

\title[Moduli difference of inverse logarithmic coefficients]{Moduli difference of inverse logarithmic coefficients of univalent functions}

\author[Vasudevarao Allu]{Vasudevarao Allu}
\address{Vasudevarao Allu, School of Basic Sciences, Indian Institute of Technology Bhubaneswar,
	 Odisha, India.}
\email{avrao@iitbbs.ac.in}
\author[Amal Shaji]{Amal Shaji}
\address{Amal Shaji, School of Basic Sciences, Indian Institute of Technology Bhubaneswar,
Bhubaneswar-752050, Odisha, India.}
\email{amalmulloor@gmail.com}

\begin{abstract}
Let $f$ be analytic in the unit disk and $\mathcal{S}$ be the subclass of
normalized univalent functions with $f(0) = 0$, and $f'(0) = 1$. Let $F$ be the inverse function of $f$, given by $F(w)=w+\sum_{n=2}^{\infty}A_nw^n$ defined on some disk $|w|\le r_0(f)$. The  inverse logarithmic
 coefficients $\Gamma_n$, $n \in \mathbb{N}$, of $f$ are defined by the equation $ \log(F(w)/w)=2\sum_{n=1}^{\infty}\Gamma_{n}w^{n},\,|w|<1/4.$ In this paper, we find the 
sharp upper and lower bounds for moduli difference of second and first inverse logarithmic coefficients, {\em i.e.,} $|\Gamma_2|-|\Gamma_1|$ for functions in class $\mathcal{S}$ and for functions in some important subclasses of univalent functions. 
\end{abstract}
\maketitle
\section{Introduction}\label{Introduction}
Let $\mathcal{H}$ denote the class of analytic functions in the unit disk $\mathbb{D}:=\{z\in\mathbb{C}:\, |z|<1\}$. Here $\mathcal{H}$ is 
a locally convex topological vector space endowed with the topology of uniform convergence over compact subsets of $\mathbb{D}$. Let $\mathcal{A}$ denote the class of functions $f\in \mathcal{H}$ such that $f(0)=0$ and $f'(0)=1$.  Let $\mathcal{S}$ 
denote the subclass of  $\mathcal{A}$ consisting of functions which are univalent ({\em i.e., one-to-one}) in $\mathbb{D}$. 
If $f\in\mathcal{S}$ then it has the following series representation
\begin{equation}\label{S}
	f(z)= z+\sum_{n=2}^{\infty}a_n z^n, \quad z\in \mathbb{D}.
\end{equation}

A domain $D \subset \mathbb{C}$ is said to be {\em starlike} with respect to a point $z_0\in D$ if the line segment joining $z_0$ to every other point $z \in D$ lies entirely in $D$. A function $f \in \mathcal{A}$ is called {\em starlike} if $f(\mathbb{D})$ is a starlike domain with respect to the origin. The class of univalent starlike functions is denoted by $\mathcal{S}^*$. A domain $D \subset \mathbb{C}$ is said to be {\em convex} if the line segment joining any two arbitrary points of $D$ lies entirely in $D$; {\em i.e.,} if it is starlike with respect to each points of $D$.
A function $f\in\mathcal{A}$ is said to be {\em convex} in $\mathbb{D}$ if $f(\mathbb{D})$ is a convex domain. The class of all univalent convex functions is denoted by $\mathcal{C}$ (see \cite{Dur83,
vasu-book-2018}). The
classes $\mathcal{S}^*$ of starlike functions and $\mathcal{C}$ of convex functions are analytically defined respectively as

\begin{equation*}
\begin{aligned}
\mathcal{S}^*&=\left\{ f \in \mathcal{A} : {\rm Re}\left(\cfrac{zf'(z)}{f(z)}\right)>0, z \in \mathbb{D} \right\}, \\[2mm]
\mathcal{C}&=\left\{ f \in \mathcal{A} : {\rm Re}\left(1+\cfrac{zf'(z)}{f(z)}\right)>0, z \in \mathbb{D} \right\}.
\end{aligned}
\end{equation*}

\medskip

\noindent  Given $\alpha \in [0,1)$, a function $f \in \mathcal{A}$  of the form \eqref{S} is called starlike functions
of order $\alpha$, if 
\begin{equation}\label{staralpha}
{\rm Re} \left( \cfrac{zf'(z)}{f(z)}\right)> \alpha, \quad z \in \mathbb{D}.
\end{equation}
The set of all such functions is denoted by $\mathcal{S}^*(\alpha)$. A function $f \in \mathcal{A}$ is called convex function of order $\alpha$, if
\begin{equation}\label{starconvex}
{\rm Re} \left( 1+\cfrac{zf''(z)}{f'(z)}\right)> \alpha, \quad z \in \mathbb{D}.
\end{equation}
The set of all such functions is denoted by $\mathcal{C}(\alpha)$. For $\alpha := 0$, these classes reduce to the well-known classes $\mathcal{S}^*$ and $\mathcal{C}$ , the class
of starlike functions and the class of convex functions, respectively.\\[2mm]
\noindent  A function $f \in \mathcal{A}$ of the form \eqref{S} is said to be strongly starlike of order $\alpha$,\, $(0 < \alpha \leq 1)$, if
\begin{equation}\label{stardef}
\left|\text{arg} \left(\cfrac{zf''(z)}{f'(z)}\right)\right|< \cfrac{\pi \alpha}{2}, \quad z \in \mathbb{D}.
\end{equation}
The set of all such functions is denoted by $\mathcal{S}^*_{\alpha}$.
\noindent A function $f \in \mathcal{A}$ of the form \eqref{S} belongs to $\mathcal{C}_{\alpha}$, the class strongly convex function of order $\alpha$,\,$(0 < \alpha \leq 1)$, if
\begin{equation}\label{convexdef}
\left|\text{arg} \left(1+\cfrac{zf''(z)}{f'(z)}\right)\right|< \cfrac{\pi \alpha}{2}, \quad z \in \mathbb{D}.
\end{equation}
\noindent The notion of strongly starlike functions was introduced by Stankiewicz \cite{Stankiewicz1} and independently by Brannan and Kirwan\cite{Brannan}. An external geometric characterisation of strongly starlike functions has been proposed by Stankiewicz \cite{Stankiewicz2}. Brannan and Kirwan\cite{Brannan} have obtained  a geometrical condition called $\delta-$ visibility which is sufficient for functions to be starlike.\\

In 1985, de Branges \cite{DeB} solved the famous Bieberbach conjecture, by showing that if $f \in \mathcal{S}$ of the form \eqref{S}, then $|a_n| \le n$ for $n \geq 2$ with equality holds for Koebe function $k(z):=z/(1-z)^2$ or its rotations. It was therefore natural to ask if for $f \in \mathcal{S}$, the inequality $||a_{n+1}| -|a_n|| \leq 1$
is true when $n \geq 2$. This problem was first studied
by Goluzin \cite{Gol46} with an aim to solve the Bieberbach conjecture. In 1963, Hayman \cite{Hay63} proved that $\big||a_{n+1}|-|a_n|\big |\le A$ for $f \in \mathcal{S}$, where $A \ge 1$
is an absolute constant and the best known estimate as of now is $3.61$ due to Grinspan \cite{Gri76}. On the other hand, for the class $\mathcal{S}$, the sharp bound is known only for $n=2$ (see \cite[Theorem~3.11]{Dur83}), namely
$$
-1\leq|a_3|-|a_2|\leq 1.029\ldots.
$$
Similarly, for functions
$f \in \mathcal{S}^*$, Pommerenke \cite{pom} has conjectured that $||a_{n+1}| - |a_n|| \leq 1$ which was proved later in 1978 by Leung \cite{leung}. For convex functions, Li and Sugawa \cite{LS17}  investigated the sharp upper bound of $|a_{n+1}|-|a_n|$ for $n\ge 2$, and sharp lower bounds for $n=2,3$. 

\medskip

For $f\in \mathcal{S}$, denote by $F$ the inverse of $f$ given by
\begin{equation}\label{InverseFunction}
F(w)=w+\sum_{n=2}^{\infty}A_nw^n,
\end{equation}
valid on some disk $|w|\le r_0(f)$.
Since $f(f^{-1}(w))=w$, by equating the coefficients, we can easily obtain
\begin{equation}\label{A2}
	A_2=-a_2 \text{ and } A_3=2a_2^2-a_3.
\end{equation}
The inverse functions are studied by several authors in different perspective (see, for instance, \cite{vasu-book-2018, SimT20} and reference therein). Recently, Sim and Thomas \cite{SimT20,ST21} obtained sharp upper and lower bounds on the difference of the moduli of successive inverse coefficients for the subclasses of univalent functions.\\[2mm]
The {\it Logarithmic coefficients} $\gamma_{n}$ of $f\in \mathcal{S}$ are defined by,
\begin{equation}\label{amal-1}
	F_{f}(z):= \log\frac{f(z)}{z}=2\sum\limits_{n=1}^{\infty}\gamma_{n}z^{n}, \quad z \in \mathbb{D}.
\end{equation}
The logarithmic coefficients $\gamma_{n}$ play a central role in the theory of univalent functions. A very few exact upper bounds for $\gamma_{n}$ seem to have been established. The significance of this problem in the context of Bieberbach conjecture was pointed by Milin\cite{milin} in his conjecture. Milin \cite{milin} has conjectured that for $f\in \mathcal{S}$ and $n\ge 2$, 
$$\sum\limits_{m=1}^{n}\sum\limits_{k=1}^{m}\left(k|\gamma_{k}|^{2}-\frac{1}{k}\right)\le 0,$$
which led De Branges, by proving this conjecture, to the proof of Bieberbach conjecture  \cite{DeB}. For the Koebe function $k(z)=z/(1-z)^{2}$, the logarithmic coefficients are $\gamma_{n}=1/n$. Since the Koebe function $k$ plays the role of extremal function for most of the extremal problems in the class $\mathcal{S}$, it is expected that $|\gamma_{n}|\le1/n$ holds for functions in $\mathcal{S}$. But this is not true in general, even in order of magnitude. Recently, various authors have taken an interest in the examination of logarithmic coefficients within the class $\mathcal{S}$ and its subclasses (see\cite{vasu2017,ALT,Thomas-2016}).


\medskip
The notion of inverse logarithmic coefficients, {\em i.e.}, logartithmic coefficients of inverse of $f$, has been proposed by Ponnusamy {\it et al.} \cite{samyinverselog}. The {\it inverse logarithmic
 coefficients} $\Gamma_n$, $n \in \mathbb{N}$, of $f$ are defined by the equation
\begin{equation}\label{Gamma}
	F_{f^{-1}}(w):= \log\frac{f^{-1}(w)}{w}=2\sum\limits_{n=1}^{\infty}\Gamma_{n}w^{n}, \quad |w|<1/4.
\end{equation}
By differentiating \eqref{Gamma} together with \eqref{A2}, we obtain
\begin{equation}\label{Gammaaa}
\begin{aligned}
& \Gamma_{1}=-\cfrac{1}{2}\,a_2,\\[2mm]
& \Gamma_2=-\cfrac{1}{2}\, a_3+\cfrac{3}{4}\, a_2^2.
\end{aligned}
\end{equation}
In 2018, Ponnusamy {\it et al.} \cite{samyinverselog} obtained the sharp upper bound  for the logarithmic inverse coefficients for the class $\mathcal{S}$. In fact Ponnusamy {\it et al.} \cite{samyinverselog} have proved that when $f\in \mathcal{S}$, $$|\Gamma_n| \leq \frac{1}{2n}
  \left(
  \begin{matrix}
    2n \\
    n
  \end{matrix}
  \right),
  \quad n \in \mathbb{N}
$$
and equality holds only for the Koebe function or one of its rotations. Further, Ponnusamy {\it et al.} \cite{samyinverselog} have obtained the sharp bound for the initial logarithmic inverse coefficients for some of the important geometric subclasses of $
\mathcal{S}$.\\

In 2023, Lecko and Partyka \cite{adamlog} studied the problem of finding the upper and lower bound for $|\gamma_2|-|\gamma_1|$ for functions in class $\mathcal{S}$ using  Loewner technique. Recently, Obradovi\'{c} and  Tuneski \cite{OTnew} provided a simple proof of the same problem. Also kumar and cho \cite{kumar2} obtained sharp upper and lower bounds for $|\gamma_2|-|\gamma_1|$ for functions in subclasses of class $\mathcal{S}$. \\[2mm]
\noindent  In this paper, we consider the problem of finding sharp upper and lower bound of moduli difference of second and first inverse logarithmic coefficients.

\noindent Before presenting the main results of this paper, we will discuss some important subclasses of univalent functions


\begin{definition}
A locally univalent function $f\in \mathcal{A}$ is said to belong to $\mathcal{G}(\nu)$ for some $\nu>0$, if it satisfies the condition
\begin{equation*}
{\rm Re}\bigg(1+\cfrac{zf''(z)}{f'(z)}\bigg)<1+\cfrac{\nu}{2}, \quad z\in \mathbb{D}.
\end{equation*}
\end{definition}

\noindent In 1941, Ozaki \cite{Ozaki41} introduced the class $\mathcal{G}(1)=:\mathcal{G}$ and proved that functions in $\mathcal{G}$ are univalent in $\mathbb{D}$. Later, Umezawa \cite{UME52} studied the class $\mathcal{G}$ and showed that this class contains the class of functions convex in one direction. Moreover, functions in $\mathcal{G}$ are proved to be starlike in $\mathbb{D}$ (see \cite{ponnusamy95},\cite{ponnusamy07}). Thus, the class $\mathcal{G}(\nu)$ is included in $\mathcal{S}^*$ whenever $\nu\in (0,1]$. It can be easily seen that functions in $\mathcal{G}(\nu)$ are not necessarily univalent in $\mathbb{D}$ if $\nu>1$. 

\begin{definition}
For $-1/2< \lambda\le 1$, the class $\mathcal{F}(\lambda)$ defined  by 
\begin{equation*}
	\mathcal{F}(\lambda)=\left\{f\in \mathcal{A}:\,	{\rm Re}\left(1+\frac{zf''(z)}{f'(z)}\right)>\frac{1}{2} -\lambda ~\mbox{ for $z \in \mathbb{D}$}\right\}.
\end{equation*}
\end{definition}
\noindent We note that clearly $\mathcal{F}(1/2)=:\mathcal{C}$ is the usual class of convex functions. Moreover, for $\lambda=1$, we obtain the class $\mathcal{F}(1)=:\mathcal{C}(-1/2)$ which considered by many researcher in the recent years (see \cite{AV19,Vibhuti,Obradovic15, PSY14}). Also, functions in $\mathcal{C} (-1/2)$ are not necessarily starlike but are convex in some direction and so are close-to-convex. Here, we recall that
a function $f \in \mathcal{A}$ is called close-to-convex if $f(\mathbb{D})$ is close-to-convex domain,{\em  i.e.} the complement of $f(\mathbb{D})$ in $\mathcal{C}$ is the union of closed half lines with pairwise disjoint
interiors. Pfaltzgraff {\it et al.} \cite{pft}  has proved that $F(\lambda)$  
contains non-starlike functions for all $ 1/2 < \lambda \leq 0$. \\

\noindent Next we will discuss about the family of spirallike functions.
\begin{definition}
The family $\mathcal{S}_\gamma(\alpha)$ of $\gamma$-spirallike functions of order $\alpha$ is defined by
$$
\mathcal{S}_\gamma(\alpha)=\left\{f \in \mathcal{A}: \operatorname{Re}\left(e^{-i \gamma} \frac{z f^{\prime}(z)}{f(z)}\right)>\alpha \cos \gamma \text { for } z \in \mathbb{D}\right\},
$$
where $\alpha \in[0,1)$ and $\gamma \in(-\pi / 2, \pi / 2)$.
\end{definition}
\noindent Each function in $\mathcal{S}_\gamma(\alpha)$ is univalent in $\mathbb{D}$ (see \cite{16}). Clearly, $\mathcal{S}_\gamma(\alpha) \subset \mathcal{S}_\gamma(0) \subset \mathcal{S}$ whenever $0 \leq \alpha<1$. Functions in $\mathcal{S}_\gamma(0)$ are called $\gamma$-spirallike, but they do not necessarily belong to the starlike family $\mathcal{S}^*$. The class $\mathcal{S}_\gamma(0)$ was introduced by Špaček \cite{27}. Moreover, $\mathcal{S}_0(\alpha)=: \mathcal{S}^*(\alpha)$ is Robertson's class of functions that are starlike functions of order $\alpha$, and $\mathcal{S}^*(0)=\mathcal{S}^*$ is the class of starlike functions. The class $\mathcal{S}^*(\alpha)$ is meaningful even if $\alpha<0$, although univalency will be destroyed in this situation. \\[2mm]
We consider another family of functions that includes the class of convex functions as a proper
subfamily.
\begin{definition}
The family $\mathcal{C}_\gamma (\alpha )$ of  $\gamma$-convex functions of order $\alpha$ is defined by
$$
{\mathcal C}_\gamma (\alpha)=\left\{f\in \mathcal{A}:{\rm Re } \bigg( e^{-i\gamma }\left ( 1+\frac{zf''(z)}{f'(z)}\right )\bigg)>\alpha \cos \gamma\right\}
$$
where $0\le \alpha <1$ and $-\pi/2<\gamma<\pi/2$.
\end{definition} 
\noindent We may set ${\mathcal C}_0(\alpha)=:{\mathcal C}(\alpha)$ which consists of the normalized
convex functions of order $\alpha$. Function in ${\mathcal C}_\gamma (0)=:{\mathcal C}_{\gamma}$ need not be
univalent in $\mathbb{D}$ for general values of $\gamma $
$(|\gamma|<\pi/2)$. For example, the function
$f(z)=i(1-z)^i-i$ is known to belong to
${\mathcal C}_{\pi/4}\backslash {\mathcal S}$. Robertson
\cite{Robertson-69} has shown  that $f\in\mathcal{C}_{\gamma}$ is
univalent if $0<\cos \gamma\leq 0.2315\cdots$. Finally, Pfaltzgraff \cite{Pfaltzgraff} has shown that
$f\in\mathcal{C}_{\gamma}$ is univalent whenever $0<\cos \gamma\leq
1/2$. This settles the improvement of the range of $\gamma$ for which $f\in\mathcal{C}_{\gamma}$
is univalent. On the other hand,  in \cite{SinghChic-77} it was also shown that
functions in ${\mathcal C}_\gamma$ which satisfy $f''(0)=0$ are
univalent for all real values of $\gamma$ with $|\gamma|<\pi /2$.
For the recent study of the class for particular values of $\alpha$ and $\gamma$, we refer to \cite{APS19}.
\section{Preliminary Results}

\begin{theorem}{\textbf{Fekete-Szegö Theorem}}\cite{feketo}:
 If $f \in \mathcal{S}$ of the form \eqref{S}, then 
\begin{equation*}
|a_3-\mu a_2^2|\le\left \{
	\begin{array}{ll}
		4\mu-3, & {\mbox{ if }} \mu \geq 1,
		\\[2mm]
		1+2e^{-2\mu/(1-\mu)}, & {\mbox{ if }} 0 < \mu < 1,\\[2mm]
		3-4\mu, &{\mbox{ if}} \,\,   \mu \leq 0,
	\end{array}
	\right.
\end{equation*}
This bound is sharp for each $\mu$.
\end{theorem}
\noindent The functional $|a_3-\mu a_2^2|$ is well-known as the
Fekete-Szegö functional.\\

\noindent Let $\mathcal{P}$ denote the class of all analytic functions $p$ having positive real part in $\mathbb{D}$, of the form
\begin{equation}\label{p}
p(z)=1+c_1z+c_2z^2+\cdots.
\end{equation}
A member of $\mathcal{P}$ is called a {\em Carath\'{e}odory function}. It is known that $|c_n|\le 2$ for a function $p\in \mathcal{P}$ and for all $n\ge 1$ (see \cite{Dur83}).

\medskip

To prove our results, we need the following lemma.

\medskip

\begin{lemma}\cite{SimT20}\label{lemma2}
Let $B_1$, $B_2$, and $B_3$ be numbers such that $B_1>0$, $B_2\in \mathbb{C},$ and $B_3\in \mathbb{R}$. Let $p\in \mathcal{P}$ be of the form \eqref{p}. Define $\Psi_+(c_1,c_2)$ and $\Psi_-(c_1,c_2)$ by
$$
\Psi_+(c_1,c_2)=|B_2c_1^2+B_3c_2|-|B_1c_1|,
$$
and 
$$
\Psi_-(c_1,c_2)=-\Psi_+(c_1,c_2).
$$
Then
\begin{equation}\label{B+}
\Psi_+(c_1,c_2)\le\left \{
	\begin{array}{ll}
		|4B_2+2B_3|-2B_1, & {\mbox{ if }} |2B_2+B_3|\ge |B_3|+B_1,
		\\[5mm]
		2|B_3|, & {\mbox{ otherwise}},
	\end{array}
	\right.
\end{equation}
and
\begin{equation}\label{B-}
	\Psi_-(c_1,c_2)\le\left \{
	\begin{array}{ll}
		2B_1-B_4, & {\mbox{ if }} B_1\ge B_4+2|B_3|,
		\\[5mm]
		2B_1\sqrt{\cfrac{2|B_3|}{B_4+2|B_3|}}, & {\mbox{ if }} B_1^2\le 2|B_3|(B_4+2|B_3|),\\[5mm]
		2|B_3|+\cfrac{B_1^2}{B_4+2|B_3|}, &{\mbox{ otherwise}},
	\end{array}
	\right.
\end{equation}
where $B_4=|4B_2+2B_3|$. All inequalities in \eqref{B+} and \eqref{B-} are sharp.
\end{lemma}

\medskip

Our main aim of this paper is to estimate the sharp Lower and upper  bounds of  $|\Gamma_2|-|\Gamma_1|$ for functions $f$ belong to $\mathcal{S}$, $\mathcal{S}^*$, $\mathcal{C}$, $\mathcal{S}^*_{\alpha}$, $\mathcal{C}_{\alpha}$, $\mathcal{S}^*(\alpha)$, $\mathcal{C}(\alpha)$, $\mathcal{G}(\nu)$,  $\mathcal{F}_0(\lambda)$,  $\mathcal{S}^*_{\gamma}(\alpha)$ and $\mathcal{C}_{\gamma}(\alpha)$.

\medskip

 \section{Main Results}
 We now state our first main result which provides sharp bounds for $|\Gamma_2|-|\Gamma_1|$ when $f$ belongs to the class $\mathcal{S}$.
 \begin{theorem}\label{classS}
 Let $f \in \mathcal{S}$ of the form \eqref{S}, then 
 $$
 |\Gamma_2|-|\Gamma_1| \leq \cfrac{1}{2}
 $$
 and 
 \begin{equation*}\label{BBB}
|\Gamma_2|-|\Gamma_1| \geq \left \{
	\begin{array}{ll}
		-\cfrac{1}{2} & {\mbox{ if }} |a_2|\leq 1,
		\\[3mm]
		-0.6353..., & {\mbox{ if }} |a_2|>1.
	\end{array}
	\right.
\end{equation*}
 \end{theorem}
 \begin{proof}
 Let $f \in \mathcal{S}$ of the form \eqref{S}. Then by \eqref{Gammaaa}, we have
\begin{equation}\label{ga}
|\Gamma_2|-|\Gamma_1|=\cfrac{1}{2}\left(|a_3-\cfrac{3}{2}a_2^2|-|a_2|\right).
\end{equation}
From the Bieberbach conjecture, for all $f \in \mathcal{S}$, we have $\cfrac{1}{2}|a_2| \leq 1$, and so by \eqref{ga}

\begin{equation}
\begin{aligned}
2(|\Gamma_2|-|\Gamma_1|)&\leq \left|a_3-\cfrac{3}{2}a_2^2\right|-\cfrac{1}{2}\left|a_2^2\right|\\[2mm]
&\leq \left| a_3-\cfrac{3}{2}a_2^2+\cfrac{1}{2}a_2^2 \right| \\[2mm]
&\leq \left| a_3-a_2^2\right|\\[2mm]
&\leq 1,
\end{aligned}
\end{equation}
and hence we have
$$
|\Gamma_2|-|\Gamma_1| \leq \cfrac{1}{2}
$$

\noindent Now we obtain the lower bound for $|\Gamma_2|-|\Gamma_1|$.
We have 
$$
|\Gamma_1|-|\Gamma_2| = \left(\cfrac{1}{2}\left|a_2\right|-\left|a_3-\cfrac{3}{2}a_2^2\right|\right).
$$
If $|a_2|\leq 1$, then clearly $|\Gamma_1|-|\Gamma_2| \leq 1/2$. If $|a_2|> 1$, then
\begin{equation}\label{fek}
\begin{aligned}
2(|\Gamma_1|-|\Gamma_2|) &= \left|a_2\right|-\left|a_3-\cfrac{3}{2}a_2^2\right|\\
&\leq \left|a_2\right|^2-\left|a_3-\cfrac{3}{2}a_2^2\right|\\
&\leq \left| a_3-\cfrac{3}{2}a_2^2+a_2^2\right|\\
&=\left| a_3-\frac{1}{2}a_2^2\right|
\end{aligned}
\end{equation}
By taking $\mu=1/2$ in Fekete-Szegö Theorem and using \eqref{fek}, we obtain 
$$
|\Gamma_1|-|\Gamma_2|\leq  \cfrac{1+2e^{-2}}{2}=0.6353...
$$
For the upper bound, equality holds for the Koebe function $k(z)=z/(1-z)^2$. For the lower bound, when $|a_2|\leq 1$, equality holds for the function
$$
f(z)=\cfrac{z}{1+z+z^2}.
$$
For $|a_2| > 1$, whether the lower bound $-0.6353...$ is
the best possible is an open problem.

 \end{proof}
\noindent Next, we obtain the sharp lower and upper bounds for $|\Gamma_2|-|\Gamma_1|$ when $f$ belongs to the class  $\mathcal{S}^*_{\alpha}$.
\begin{theorem}\label{thmm}
Let $\alpha \in (0,1]$. If $f\in\mathcal{S}^*_{\alpha}$ given by \eqref{S}, then the following sharp inequalities holds.
\begin{equation}\label{thm2}
-\cfrac{\alpha}{\sqrt{1+3\alpha}}\leq |\Gamma_2|-|\Gamma_1| \leq \cfrac{\alpha}{2}.
\end{equation}

\end{theorem}

\begin{proof}
Fix $\alpha \in (0,1]$ and let $f\in \mathcal{S}^*_{\alpha}$ be of the form $(1.1)$. Then by \eqref{stardef}, we have
\begin{equation}\label{3.3.1}
\cfrac{zf'(z)}{f(z)}=(p(z))^\alpha
\end{equation}
for some $p \in \mathcal{P}$ of the form \eqref{p}. By comparing the coefficient of powers of $z$ on both sides of \eqref{3.3.1}, we obtain
	\begin{equation}\label{3.3.2}
		 a_2=\alpha c_1 \,\,\, \text{and}\,\,\,  a_3=\frac{\alpha}{4}(2c_2+(3\alpha-1)c_1^2).
	\end{equation}
\noindent Now using \eqref{Gammaaa} together with \eqref{3.3.2}, we have 
\begin{equation}\label{lambda1}
|\Gamma_2|-|\Gamma_1| = \cfrac{\alpha}{4}\left(|B_3 c_2+B_2 c
_1^2|-|B_1 c_1|\right)= \cfrac{\alpha}{4}\,\Psi_+(c_1,c_2),
\end{equation}
where  
$$
	B_1:=2, B_2:=-\frac{1}{2}(1+3\alpha), \mbox{ and } B_3:=1.
$$	
\noindent For the upper bound, we see that the condition $|2B_2+B_3| \leq |B_3|+B_1$ is equivalent
to $3 \alpha \leq 3$, which is true only if $\alpha=1$. Therefore, by Lemma \ref{lemma2},	we have
$$
|\Gamma_2|-|\Gamma_1| \leq \cfrac{\alpha}{2}.
$$
\noindent Equality holds for the function $f \in \mathcal{A}$ given by \eqref{3.3.1}, where
$p(z)=(1+z^2)/(1-z^2)$. Then $c_1=0$ and $c_2=2$, so by \eqref{3.3.2}, $a_2=0$ and $a_3=\alpha$, and therefore by \eqref{Gammaaa},
$$|\Gamma_2|-|\Gamma_1|=\cfrac{\alpha}{2}.$$
\noindent For the lower bound, since $B_4=|4B_2+2B_3|=6\alpha$, it is easy to see that the inequality $B_1 \geq B_4+2|B_3|$ doesnot holds and the inequality $B_1^2 \leq 2|B_3|(B_4+2|B_3|)$ holds for $0 < \alpha \leq 1$. Hence by Lemma \ref{lemma2}, we obtain
\begin{equation*}
|\Gamma_2|-|\Gamma_1| \geq -\cfrac{\alpha}{\sqrt{1+3\alpha}}\,.
\end{equation*}
\noindent We finally show that the left hand side inequality in \eqref{thm2} is sharp. Equality holds
for the function $f \in \mathcal{A}$ given by \eqref{3.3.1} with
$$
p(z)=\cfrac{1+2Az+z^2}{1-z^2}, \quad z \in \mathbb{D},
$$
where $A=1/\sqrt{1+3\alpha}$, which completes the proof of the theorem.
	\end{proof} 
If we put $\alpha=0$ in Theorem \ref{thmm}, then we obtain the following result for the class of starlike functions.

\begin{corollary}
	For every $f\in \mathcal{S}^*$ of the form \eqref{S}, we have
	\begin{equation}\label{C}
	-\cfrac{1}{2}\le	|\Gamma_2|-|\Gamma_1|\le \cfrac{1}{2
	}.
	\end{equation}
Both inequalities are sharp.
\end{corollary}	
\noindent Next, we obtain the sharp lower and upper bounds for $|\Gamma_2|-|\Gamma_1|$ when $f$ belongs to the class  $\mathcal{C}_{\alpha}$.	
	
\begin{theorem}
Let $\alpha \in (0,1]$. If $f\in\mathcal{C}_{\alpha}$ given by \eqref{S}, then the following sharp inequalities holds;
\begin{equation}\label{thm1}
|\Gamma_2|-|\Gamma_1| \leq \frac{\alpha}{12}
\end{equation}
and 
\begin{equation}\label{lll}
|\Gamma_2|-|\Gamma_1| \geq \left \{
	\begin{array}{lll}
		  -\cfrac{1}{4}\alpha(2-\alpha)  & {\mbox{ if }} \,\,  0 < \alpha \leq \cfrac{1}{3}, \\[5mm]
		 -\cfrac{\alpha}{6} \left( \cfrac{6\alpha+3}{6\alpha+14} \right)    &  {\mbox{ if }} \,\, \cfrac{1}{3} < \alpha < \cfrac{5}{6}, \\[5mm]
		 -\cfrac{\alpha}{\sqrt{4+6\alpha}} &  {\mbox{ if }} \,\, \cfrac{5}{6} \leq \alpha \leq 1.
	\end{array}
	\right.
\end{equation}
\end{theorem}

\begin{proof}
Fix $\alpha \in (0,1]$ and let $f\in \mathcal{K}_{\alpha}$ be of the form \eqref{S}. Then by \eqref{convexdef},
\begin{equation}\label{3.1.1}
1+\cfrac{zf''(z)}{f'(z)}=(p(z))^\alpha
\end{equation}
for some $p \in \mathcal{P}$ of the form \eqref{p}. By comparing the coefficients of powers of $z$ on both the sides of \eqref{3.1.1}, we get
	\begin{equation}\label{3.12}
a_2=\cfrac{1}{2}\,\alpha c_1 \,\,\, \text{and} \,\,\, a_3=\frac{\alpha}{12}(2c_2+(3\alpha-1)c_1^2).
	\end{equation}
Now using \eqref{Gammaaa} together with \eqref{3.12}, we obtain
\begin{equation*}
\begin{aligned}
|\Gamma_2|-|\Gamma_1| &= \cfrac{1}{48}\alpha|4c_2-(2+3\alpha)c_1^2|-\cfrac{1}{4}\alpha |c_1| \\[2mm]
&=\cfrac{\alpha}{12}\left(|B_3 c_2+B_2 c_1^2|-|B_1 c_1|\right)= \cfrac{\alpha}{12} \, \Psi_+(c_1,c_2) \\[2mm]
\end{aligned}
\end{equation*}
where  
$$
B_1:=3, B_2:=-\frac{1}{4}(2+3\alpha), \mbox{ and } B_3:=1.
$$	
It is easy to see that the condition $|2B_2+B_3| > |B_3|+B_1$ holds for $0 < \alpha \leq 1$, Therefore by Lemma \ref{lemma2}, we get 	
$$
	|\Gamma_2|-|\Gamma_1| \leq \cfrac{\alpha}{6}.
$$
Equality holds for the function $f$ is of the form \eqref{3.1.1}, with 
$$p(z)=\cfrac{1+z^2}{1-z^2}, \quad z \in \mathbb{D}.$$	
\noindent We now consider the lower bound. Then
\begin{equation}\label{ppppp}
|\Gamma_1|-|\Gamma_2|=\cfrac{\alpha}{12} \, \Psi_-(c_1,c_2),
\end{equation}
where $\Psi_-(c_1,c_2)=-\Psi_+(c_1,c_2)$. Since $B_4=3\alpha$, the inequality $B_1 \geq B_4+2|B_3|$ holds for $0< \alpha \leq 1/3$. Observe that $2|B_3|(B_4+2|B_3|)-B_1^2=6\alpha-5 \geq 0$ is true for $5/6 \leq \alpha \leq 1$. Hence, Lemma \ref{lemma2} gives
\begin{equation}\label{oo}
\Psi_-(c_1,c_2) \leq \left \{
	\begin{array}{lll}
		  \cfrac{1}{4}\alpha(2-\alpha)  & {\mbox{ if }} \,\,  0 < \alpha \leq \cfrac{1}{3},
		\\[4mm]
		   \cfrac{\alpha}{6} \left( \cfrac{6\alpha+3}{6\alpha+14} \right)    &  {\mbox{ if }} \,\, \cfrac{1}{3} < \alpha < \cfrac{5}{6}, \\[4mm]
		 \cfrac{\alpha}{\sqrt{4+6\alpha}} &  {\mbox{ if }} \,\, \cfrac{5}{6} \leq \alpha \leq 1.
	\end{array}
	\right.
\end{equation}
From \eqref{ppppp} and \eqref{oo}, we obtain the inequality in \eqref{lll} as required.\\

\noindent We finally show that the inequalities in \eqref{lll} are sharp. When $\alpha \in (0,1/3]$, equality holds for the function $f \in \mathcal{A}$ given by \eqref{3.1.1} with $p(z)=1+z/1-z$. In this case $a_2=\alpha$ and $a_3=\alpha^2$ and so by \eqref{Gammaaa}, $\Gamma_1=-\alpha/2$  and $\Gamma_2=-\alpha^2/4$. When $\alpha \in (1/3,5/6)$, equality holds for the function $f \in \mathcal{A}$ given by \eqref{3.1.1} with 
$$
p(z)=\cfrac{1+2Bz+z^2}{1-z^2},
$$
where $B=3/(3+2\alpha)$.  When $ \alpha \in (5/6,1]$, equality holds for the function $f \in \mathcal{A}$ given by \eqref{3.1.1} with 
$$
p(z)=\cfrac{1+2Cz+z^2}{1-z^2},
$$
where $C=3/\sqrt{3+2\alpha}$. This completes the proof of the theorem. 
\end{proof}	
	
If we put $\alpha=0$ in Theorem \ref{thm1}, then we obtain the following result for the class of convex functions.

\begin{corollary}
	For every $f\in \mathcal{C}$ of the form \eqref{S}, we have
	\begin{equation}\label{C}
	-\cfrac{1}{\sqrt{10}}\le	|\Gamma_2|-|\Gamma_1|\le \cfrac{1}{6}.
	\end{equation}
Both inequalities are sharp.
\end{corollary}	
	
\begin{theorem}\label{1}
Let $f\in \mathcal{G}(\nu)$ of the form $(1.1)$. Then the following sharp estimate holds.
\begin{equation}
|\Gamma_2|-|\Gamma_1| \leq \frac{\nu}{12}
\end{equation}

\begin{equation}\label{inverse1}
	|\Gamma_2|-|\Gamma_1|\ge \left \{
	\begin{array}{ll}
		\cfrac{\nu}{12}\left( \cfrac{10\nu+34}{5\nu+8} \right) & {\mbox{ if }} \,\,  \cfrac{1}{5} \leq \nu \leq 1,
		\\[4mm]
		\cfrac{\nu}{\sqrt{5\nu+8}}, & {\mbox{ if }} \,\, \cfrac{1}{5} \leq \nu \leq 1.
	\end{array}
	\right.
\end{equation} 
 \end{theorem}

\begin{proof}
Let $f\in \mathcal{G}(\nu)$. Then there exists a function $p(z)$ of the form \eqref{p} satisfying  the condition,
 \begin{equation}\label{gp(z)}
 	p(z)=\cfrac{1}{\nu} \bigg(\nu- \cfrac{2zf''(z)}{f'(z)}\bigg).
 \end{equation}
Then, on equating coefficients of powers of $z$ in \eqref{gp(z)}, we have
\begin{equation}\label{g1}
	a_2=-\cfrac{\nu c_1}{4} \,\,\, \text{and} \,\,\,
    a_3=\cfrac{\nu^2 c_1^2-2\nu c_2}{24}.
\end{equation}

\noindent Using \eqref{Gammaaa} and \eqref{g1}, we obtain
\begin{equation}\label{A1}
|\Gamma_2|-|\Gamma_1| = \cfrac{\nu}{24}\left(|B_3 c_2+B_2 c_1^2|-|B_1 C_1|\right)=\cfrac{\nu}{24}\,\,\Psi_+(c_1,c_2)
\end{equation} 
 where $$
	B_1:=3, B_2:=\frac{5}{8}\nu, \mbox{ and } B_3:=1.
	$$
	
\noindent Also note that $B_4=|4B_2+2B_3|=(5\nu+4)/4$.\\[2mm]
\noindent For the upper bound, we see that the condition $|2B_2+B_3| \leq |B_3|+B_1$ is equivalent
to $5 \nu \leq 12$, which is not true for $0< \nu \leq 1$. Therefore	
$$
	|\Gamma_2|-|\Gamma_1| \leq \cfrac{\nu}{12}\,.
$$
It is easy to see that equality holds when $f \in \mathcal{G}(\nu)$ defined by \eqref{gp(z)} with 
	$p(z)=(1+z^2)/(1-z^2)$. \\[2mm]
\noindent We now proceed to prove the lower bound for $|\Gamma_2|-|\Gamma_1|$. It is easy to show that the condition $B_1 \geq B_4 + 2|B_3|$ doesnot hold for $0\leq \nu \leq 1$, but the condition $B_1^2 \leq 2|B_3|(B_4 + 2|B_3|)$ holds for $\nu \geq 1/5$.  Thus by Lemma \ref{lemma2}, we obtain
$$
\Psi_-(c_1,c_2)\le \left \{
\begin{array}{ll}
     \cfrac{10\nu+34}{5\nu+8}, & {\mbox{ if }} 0 < \nu < 1/5,\\[4mm]
	 \cfrac{12}{\sqrt{5\nu+8}}, & {\mbox{ if }} 1/5 \leq \nu \leq 1.
\end{array}
\right.
$$\\
\noindent By substituting the above inequality in \eqref{A1}, we get the inequality in Theorem \ref{1}. \\[1mm]
For $0 < \nu < 1/5$, equality holds for the function $f_2$ is of the form \eqref{gp(z)} with 
$$
p(z)=\cfrac{1-z^2}{1-2tz+z^2},
$$
where $t=6/(5\nu+8)$. For $1/5 \leq \nu \leq 1$, equality holds for the function $f_3$ is of the form \eqref{gp(z)} with 
$$
p(z)=\cfrac{1-z^2}{1-2sz+z^2},
$$
where $s=2/(\sqrt{5\nu+8})$. This completes the proof of the theorem.
\end{proof} 

 \begin{theorem}\label{2}
Let $1/2\le \lambda\le 1$. For every $f\in \mathcal{F}_0(\lambda)$ be of the form \eqref{S}, then the following sharp inequalities holds.

\begin{equation}
-\cfrac{1+2\lambda}{2\sqrt{5+10\lambda}} \leq |\Gamma_2|-|\Gamma_1| \leq \frac{\nu}{12}.
\end{equation}

\end{theorem}

\begin{proof}

Let  $f\in \mathcal{F}_0(\lambda)$ be of the form \eqref{S}. Then we have
	\begin{equation}\label{pf}
	1+\cfrac{zf''(z)}{f'(z)}=\bigg(\cfrac{1}{2}+\lambda\bigg)p(z)+\cfrac{1}{2}-\lambda, \quad z\in \mathbb{D}.
	\end{equation}
	By comparing the coefficients of powers of $z$ on both the sides of \eqref{pf}, we obtain
\begin{equation}\label{eqa2c2}
	a_2=\cfrac{(1+2\lambda)c_1}{4} \,\,\, \text{and} \,\,\,
    a_3=\cfrac{(1+2\lambda)(2c_2+(1+2\lambda)c_1^2)}{24}.
\end{equation}

\noindent Now using \eqref{Gammaaa} together with \eqref{eqa2c2}, we have 

\begin{equation}\label{lambda1}
|\Gamma_2|-|\Gamma_1| = \cfrac{(1+2\lambda)}{24}\left(|B_3 c_2+B_2 c_1^2|-|B_1 c_1|\right)= \cfrac{(1+2\lambda)}{24}\psi_+(c_1,c_2)
\end{equation}
where  
$$
	B_1:=3, B_2:=-\frac{5}{8}(1+2\lambda), \mbox{ and } B_3:=1.
$$

\noindent Note that the inequality $|2B_2+B_3| \geq |B_3|+B_1$ is equivalent to $1+10\lambda \geq 16$, which is not true for $1/2 \leq \lambda \leq 1$. So using Lemma together with \eqref{lambda1}, we can conclude that 
$$|\Gamma_2|-|\Gamma_1| \leq \cfrac{1+2\lambda}{12},$$
and equality holds for the function $f_4$ given by
 \begin{equation}\label{f2}
 	f_2(z)=z+\cfrac{(2\lambda+1)}{6}\,z^3+\cdots
 \end{equation}
\noindent We now consider the lower bound. Then
\begin{equation}\label{lambda2}
|\Gamma_1|-|\Gamma_2| = \cfrac{(1+2\lambda)}{24} \psi_-(c_1,c_2)
\end{equation}
where $\psi_-(c_1,c_2)=-\psi_+(c_1,c_2)$. Since $B_4=(1+10\lambda)/2$, it is easy to see that the inequality $B_1\geq B_4+2|B_3|$ is not true. Further we have,
$$
2|B_3|(B_4+2|B_3|)-B_1^2=10\lambda -4 \geq 0.
$$
Hence by \eqref{lambda2} and Lemma \ref{lemma2} we get 
$$
|\Gamma_1|-|\Gamma_2| \leq \cfrac{1}{2}\left(\cfrac{1+2\lambda}{\sqrt{5+10\lambda}}\right),
$$
and equality holds for the function $f \in \mathcal{F}_0(\lambda)$ satisfying \eqref{pf} with 
$$
p(z)=\cfrac{1+2tz+z^2}{1-z^2},
$$
where $t=2/\sqrt{5+10\lambda}$. This completes the proof of the theorem.
\end{proof}

\noindent Next, we obtain the sharp lower and upper bounds for $|\Gamma_2|-|\Gamma_1|$ when $f$ belongs to the class  $\mathcal{S}^*_{\gamma}(\alpha)$.

\begin{theorem}\label{4444}
Let $-\pi/2< \gamma< \pi/2$ and $0\le \alpha <1$. For every $f\in \mathcal{S}^*_{\gamma}(\alpha)$ of the form \eqref{S}, the following sharp inequalities holds.
$$
-\cfrac{(1-\alpha)\cos \gamma}{\sqrt{|\eta|+1}} \leq|\Gamma_2|-|\Gamma_1| \leq \cfrac{(1-\alpha)\cos \gamma}{2}
$$
\noindent where $\eta=4(1-\alpha)\mu-1$ with $\mu=e^{i\gamma}\cos \gamma$.
\end{theorem}

\begin{proof}
Let $f\in \mathcal{S}^*_{\gamma}(\alpha)$, then there exists a function $p\in \mathcal{P}$ such that
\begin{equation}\label{spiral}
p(z)=\cfrac{1}{1-\alpha}\bigg( \cfrac{1}{\cos \gamma}\bigg(e^{-i\gamma} \bigg(\cfrac{zf'(z)}{f(z)}\bigg)+i\sin \gamma\bigg)-\alpha\bigg).
\end{equation}
Equating the coefficients of $z^n$ on both the sides of \eqref{spiral} for $n=1,2$, we obtain
\begin{equation}\label{eqa2}
a_2=(1-\alpha)\mu c_1 \,\,\, \text{and} \,\,\,
	a_3=\cfrac{(1-\alpha)\mu}{2}\left((1-\alpha)\mu c_1^2+ c_2\right),
\end{equation}
where $\mu=e^{i\gamma}\cos \gamma$. Now using \eqref{Gammaaa} together with \eqref{eqa2}, we get

\begin{align}\label{cB2}
	|\Gamma_2|-|\Gamma_1|	&=\cfrac{(1-\alpha)\cos \gamma}{4}\left(|B_3c_2+B_2 c_1^2|-|B_1c_1|\right),
\end{align}
where
$$
B_1:=2,\,B_2:=-2(1-\alpha)\mu, \mbox{ and }B_3=1.
$$

\noindent It is easy to see that the inequality $|2B_2+B_3| \leq |B_3|+B_1$ is true for $-\pi/2< \gamma< \pi/2$ and $0\le \alpha <1$. Therefore by Lemma \ref{lemma2}, we get 
$$
|\Gamma_2|-|\Gamma_1| \leq \cfrac{(1-\alpha)\cos \gamma}{2}
$$
\noindent Note that $B_4=|8(1-\alpha)\mu-2|$. For the lower bound, it is easy to see that  the inequality $B_1 \geq B_4+2|B_3|$ doesnot holds and the inequality $B_1^2 \leq 2|B_3|(B_4+2|B_3|)$ holds for $-\pi/2< \gamma< \pi/2$ and $0\le \alpha <1$. Hence by Lemma \ref{lemma2}, we get 
\begin{equation}
|\Gamma_2|-|\Gamma_1|\geq -\cfrac{(1-\alpha)\cos \gamma}{\sqrt{|\eta|+1}}
\end{equation}
\noindent where $\eta=4(1-\alpha)\mu-1$. Equality holds for the function $f$ defined by \eqref{spiral} with
\begin{equation}\label{eqcp(z)}
p(z)=\cfrac{1+q_1(q_2+1)z+q_2z^2}{1+q_1(q_2-1)z-q_2z^2},
\end{equation}
where
$$
q_1=\cfrac{1}{\sqrt{|\tau|+1}} \mbox{ and } q_2=e^{i\arg \tau}.
$$
  This completes the proof.
\end{proof}
\noindent If we put  $\gamma=0$ in Theorem \ref{4444}, then we obtain the following result for the class of starlike function of order $\alpha$.
\begin{corollary}\label{444}
Let $0 \le \alpha < 1$ and  $f\in \mathcal{S}^*(\alpha)$ of the form \eqref{S}. Then
\begin{equation}
|\Gamma_2|-|\Gamma_1|\le \cfrac{1-\alpha}{2}
\end{equation}
and
\begin{equation}
|\Gamma_2|-|\Gamma_1|\geq \left \{
	\begin{array}{ll}
		  -\cfrac{\sqrt{1-\alpha}}{2}  & {\mbox{ if }} \,\,  0 \leq \alpha \leq \cfrac{3}{4},
		\\[4mm]
		  -\cfrac{1-\alpha}{\sqrt{4\alpha-2}}  &  {\mbox{ if }} \,\, \cfrac{3}{4} \leq \alpha < 1,
	\end{array}
	\right.
\end{equation}
\end{corollary}

\noindent Next, we obtain the sharp lower and upper bounds for $|\Gamma_2|-|\Gamma_1|$ when $f$ belongs to the class  $\mathcal{C}_{\gamma}(\alpha)$.
\begin{theorem}\label{3}
Let $-\pi/2< \gamma< \pi/2$ and $0\le \alpha <1$. For every $f\in \mathcal{C}_{\gamma}(\alpha)$ of the form \eqref{S}, the following sharp inequality holds.
	\begin{equation}\label{55555}
	|\Gamma_2|-|\Gamma_1|\le \cfrac{(1-\alpha)\cos \gamma}{6}
	\end{equation}
	and
	\begin{equation}\label{l}
|\Gamma_2|-|\Gamma_1|\ge\left \{
\begin{array}{lll}
\cfrac{(\alpha-1)\cos \gamma \,}{12}(6-|\beta|) , & {\mbox{ if }} |\beta| \le 1,\\[5mm]
\left(\cfrac{(\alpha-1)\cos \gamma}{12}\right)\left(2+\cfrac{9}{|\beta|+2}\right), & {\mbox{ if }} 1 \le |\beta| \le 5/2,\\[5mm]
\left(\cfrac{(\alpha-1)\cos \gamma}{2}\right)\left(\sqrt{\cfrac{2}{|\beta|+2}}\right), &{\mbox{ if }} |\beta| \ge 5/2.
\end{array}
\right.
\end{equation}
where $\beta=5(1-\alpha)\mu-2$ with $\mu=e^{i\gamma}\cos \gamma$.
\end{theorem}

\begin{proof}
Let $f\in \mathcal{C}_{\gamma}(\alpha)$, then there exists a function $p\in \mathcal{P}$ such that
\begin{equation}\label{convexp(z)}
p(z)=\cfrac{1}{1-\alpha}\bigg\{ \cfrac{1}{\cos \gamma}\bigg(e^{-i\gamma} \bigg(1+\cfrac{zf''(z)}{f'(z)}\bigg)+i\sin \gamma\bigg)-\alpha\bigg\}.
\end{equation}
Equating the coefficients of $z^n$ on both the sides of \eqref{convexp(z)} for $n=1,2$, we obtain
\begin{equation}\label{eqa2a3}
a_2=\cfrac{(1-\alpha)\mu}{2} c_1 \,\,\, \text{and} \,\,\,
	a_3=\cfrac{(1-\alpha)\mu}{6}\left((1-\alpha)\mu c_1^2+ c_2\right),
\end{equation}
where $\mu=e^{i\gamma}\cos \gamma$. Now using \eqref{Gammaaa}, we obtain
\begin{align}\label{cB2}
	|\Gamma_2|-|\Gamma_1|&=\cfrac{(1-\alpha)\cos \gamma}{12}\bigg(|c_2-\cfrac{5}{4}(1-\alpha)\mu c_1^2|-|3c_1|\bigg)\nonumber\\[2mm]
	&=\cfrac{(1-\alpha)\cos \gamma}{12}\bigg(|B_3c_2+B_2 c_1^2|-|B_1c_1|\bigg),
\end{align}
where
$$
B_1:=3,\,B_2:=-\cfrac{5}{4}(1-\alpha)\mu, \mbox{ and }B_3=1.
$$
For the upper bound, it is easy to see that the first condition $|2B_2+B_3|\ge |B_3|+B_1$ is not satisfied since $|5(1-\alpha)\mu-2|\le 3$. 
By using Lemma \ref{lemma2} and the equation \eqref{cB2}, we obtain that 
$$
|\Gamma_2|-|\Gamma_1| \le \cfrac{(1-\alpha)\cos \gamma}{6}.
$$
This proves the inequality \eqref{55555}.

\medskip

\noindent We next prove the lower bound in \eqref{l} by checking the condition of Lemma \ref{lemma2} for $\Psi_-(c_1,c_2)$. Note that the inequality $B_1\ge B_4+2|B_3|$ is true for $|5(1-\alpha)\mu-2|\le 1$
and the inequality $ B_1^2\le 2|B_3|(B_4+2|B_3|)$ is true for $|5(1-\alpha)\mu-2|\ge 5/2$. Thus, by applying Lemma \ref{lemma2}, we obtain
$$
\Psi_-(c_1,c_2)\le\left \{
\begin{array}{lll}
6-|\beta| , & {\mbox{ if }} |\beta| \le 1,\\[3mm]
2+\cfrac{9}{|\beta|+2}, & {\mbox{ if }} 1 \le |\beta| \le 5/2,\\[5mm]
6\left(\sqrt{\cfrac{2}{|\beta|+2}}\right), &{\mbox{ if }} |\beta| \ge 5/2.
\end{array}
\right.
$$
where $\beta=5(1-\alpha)\mu-2$. By Substituting the above inequality in \eqref{cB2} we obtain the required inequality \eqref{l}.

\medskip

\noindent Equality holds in \eqref{55555} when $f$ is defined by \eqref{convexp(z)} with $p(z)=(1+z^2)/(1-z^2)$. For $|\beta|\ge 5/2$, equality holds for the function $f$ defined by \eqref{convexp(z)} with
\begin{equation}\label{eqcp(z)}
p(z)=\cfrac{1+q_1(q_2+1)z+q_2z^2}{1+q_1(q_2-1)-q_2z^2},
\end{equation}
where $
q_1=1/\sqrt{1+|\beta|} \mbox{ and } q_2=e^{i\arg \beta}$.
\noindent For $|\beta|\le 5/4$, equality holds for the function $f$ given by \eqref{convexp(z)} where $p(z)$ is of the form \eqref{eqcp(z)} with 
 $q_1=3/2(1+|\beta|) \mbox{ and } q_2=e^{i\arg \beta}$.
\noindent This completes the proof of this theorem.
\end{proof}
If we put  $\gamma=0$ in Theorem \ref{3}, then we obtain the following result for the class of convex function of order $\alpha$.
\begin{corollary}\label{4}
Let $0 \le \alpha < 1$ and  $f\in \mathcal{C}(\alpha)$ of the form \eqref{S}. Then
\begin{equation}
|\Gamma_2|-|\Gamma_1|\le \cfrac{1}{6}(1-\alpha)
\end{equation}
and

\begin{equation}
|\Gamma_1|-|\Gamma_2|\le\left \{
\begin{array}{lllll}
\sqrt{\left(\cfrac{1-\alpha}{10}\right)}, & {\mbox{ if }} 0 \le \alpha \le \cfrac{1}{10},\\[5mm]
\cfrac{1}{60}(19-10\alpha), & {\mbox{ if }} \cfrac{1}{10} \le \alpha \le \cfrac{2}{5},\\[5mm]
\cfrac{1}{12}(1-\alpha)(3+5\alpha), &{\mbox{ if }} \cfrac{2}{5} \le \alpha \le \cfrac{3}{5},\\[5mm]
\cfrac{1}{12}(1-\alpha)(9-5\alpha), & {\mbox{ if }} \cfrac{3}{5} \le \alpha \le \cfrac{4}{5},\\[5mm]
\cfrac{(1-\alpha)(10\alpha+7)}{12(5\alpha-1)}, & {\mbox{ if }} \cfrac{4}{5} \le \alpha  \le 1.\\[5mm]
\end{array}
\right.
\end{equation}
All the inequalities are sharp.
\end{corollary}

{\bf Acknowledgement:}
The research of the first named author is supported by SERB-CRG, Govt. of India and the second named author's research work is supported by UGC-SRF.

\end{document}